\documentclass[leqno,11pt]{amsart}

\usepackage{amsmath,amstext,amssymb,amsopn,amsthm,mathrsfs}

\title[Sharp estimates of the Jacobi heat kernel]
    {Sharp estimates of the Jacobi heat kernel}
    
\author[A{.} Nowak]{Adam Nowak}
\author[P{.} Sj\"ogren]{Peter Sj\"ogren}

\address{Adam Nowak, \newline
			Institute of Mathematics,
      Polish Academy of Sciences, \newline
      \'Sniadeckich 8,
      00--956 Warszawa, Poland      
      }
\email{anowak@impan.pl}

\address{
\noindent Peter Sj\"ogren \newline
    Mathematical Sciences, 
    University of Gothenburg \newline
    Mathematical Sciences,
    Chalmers University of Technology \newline 
    SE-412 96 G\"oteborg,
    Sweden}
\email{\noindent peters@chalmers.se}

\allowdisplaybreaks
\theoremstyle{plain}
\newtheorem{thm}{Theorem}[section]
\newtheorem{bigthm}{Theorem}
\newtheorem{lem}[thm]{Lemma}

\newtheorem{cor}[thm]{Corollary}

\theoremstyle{definition}

\theoremstyle{remark}
\newtheorem*{rem*}{Remark}

\theoremstyle{plain}

\def\R{\mathbb R}
\def\N{\mathbb N}

\def\P{\mathcal P}

\def\m{\mu} 						%Jacobi measure
\def\r{\varrho}
\def\ab{\alpha,\beta}
\def\J{\mathcal J} 			%Jacobi operator

\begin{document}

\begin{abstract}
The heat kernel associated with the setting of the classical Jacobi polynomials is defined by
an oscillatory sum which cannot be computed explicitly, in contrast to the situation for the
two other classical systems of orthogonal polynomials.
We deduce sharp estimates giving the order of magnitude of this kernel, for type parameters $\ab \ge -1/2$.
As an application of the upper bound obtained, we show that the maximal operator
of the multi-dimensional Jacobi heat semigroup satisfies a weak type $(1,1)$ inequality.
We also obtain sharp estimates of the Poisson-Jacobi kernel.
\end{abstract}

\maketitle

\footnotetext{
%42C05 - Orthogonal functions and polynomials, general theory
%42C10 - Fourier series in special orthogonal functions
%35K08 - Heat kernel

	The first-named author was supported in part by MNiSW Grant N N201 417839.
}

%%%%%%%%%%%%%%%%%%%%%%%%%%%%%%%%%%%%%%%%%%%%%%%%%%%%%%%%%%%%%%%%%%%%%%%%%%%%%%%%%%%%%%%%%%%%%%%%%%%%%
\section{Introduction} \label{sec:intro}
%%%%%%%%%%%%%%%%%%%%%%%%%%%%%%%%%%%%%%%%%%%%%%%%%%%%%%%%%%%%%%%%%%%%%%%%%%%%%%%%%%%%%%%%%%%%%%%%%%%%%

Let $P_n^{\ab}$, $n=0,1,2,\ldots$, be the classical Jacobi polynomials with type 
parameters $\alpha,\beta>-1$, as defined in Szeg\"o's monograph \cite{Sz}.
The Jacobi heat kernel is given by
\begin{equation} \label{ker_ser}
G_t^{\ab}(x,y) = \sum_{n=0}^{\infty} \exp\Big({-t\, n(n+\alpha+\beta+1)}\Big) 
	\frac{P_n^{\ab}(x) P_n^{\ab}(y)}{h_n^{\ab}}, \qquad x,y \in [-1,1], \quad t>0,
\end{equation}
where $h_n^{\ab} = \int_{-1}^1 [P_n^{\ab}(x)]^2 (1-x)^{\alpha}(1+x)^{\beta}dx$ are proper normalizing
constants. The numbers $n(n+\alpha+\beta+1)$ here are the eigenvalues of the Jacobi differential operator,
and it is well known that this kernel gives the solution of the initial-value problem for
the Jacobi heat equation, as explained below.

\vspace{5pt}
Our main result reads as follows.
\begin{bigthm} \label{thm:main}
Assume that $\alpha,\beta \ge -1/2$. Given any $T>0$, there exist positive constants 
$C,c_1$ and $c_2$, depending only on $\alpha,\beta$ and $T$, such that
\begin{align*}
& \; \frac{1}{C}\big[ t + \theta \varphi \big]^{-\alpha-1/2}\, 
	\big[ t + (\pi-\theta)(\pi-\varphi)\big]^{-\beta-1/2}
	\, \frac{1}{\sqrt{t}} \exp\bigg({-c_1 \frac{(\theta-\varphi)^2}t}\bigg) \\
 \le & \; G_t^{\ab}(\cos\theta,\cos\varphi) \\
\le & \; C \big[ t + \theta \varphi \big]^{-\alpha-1/2} \,
	\big[ t + (\pi-\theta)(\pi-\varphi)\big]^{-\beta-1/2}
	\, \frac{1}{\sqrt{t}} \exp\bigg({- c_2\frac{(\theta-\varphi)^2}{t}}\bigg),
\end{align*}
for $\theta,\varphi \in [0,\pi]$ and $0<t\le T$. Moreover,
$$
C^{-1} \le G_t^{\ab}(x,y) \le C,
$$
for $x,y \in [-1,1]$ and $t \ge T$, and $G_t^{\ab}(x,y) \to 1/h_0^{\ab}$ as $t\to \infty$,
uniformly in $x,y \in [-1,1]$.
\end{bigthm}
Thus we obtain a qualitatively sharp description of the behavior of $G_t^{\ab}(x,y)$.
The restriction on $\alpha$ and $\beta$ is imposed by the methods used.
Nevertheless, it is natural to conjecture that the same bounds hold for all $\alpha,\beta > -1$.

The multi-dimensional Jacobi heat kernel is a tensor product of one-dimensional kernels, and
Theorem \ref{thm:main} provides similar bounds also in the multi-dimensional setting.
As an application, we prove that the maximal operator of the multi-dimensional Jacobi semigroup
satisfies a weak type $(1,1)$ estimate, see Theorem \ref{thm:max}. This complements
analogous results in the Hermite and Laguerre polynomial settings, which were obtained 
in dimension one by Muckenhoupt \cite{Mu}, and in arbitrary finite dimension by the
second author \cite{Sj1} and Dinger \cite{Di}, respectively.
For the Laguerre case, see also the authors' paper \cite{NoSj1}.

The heat kernels associated with the two other families of classical orthogonal polynomials
have been known explicitly for a long time. Already in 1866, Mehler \cite{Me} established the formula
$$
\sum_{n=0}^{\infty} \frac{H_n(x)H_n(y)}{2^n n!} \, r^n = \frac{1}{\sqrt{1-r^2}}
	\exp\bigg( \frac{2xyr - (x^2+y^2)r^2}{1-r^2}\bigg), \qquad |r|<1,
$$
which makes it possible to sum the heat kernel related to the Hermite polynomials $H_n$. 
In the case of the Laguerre polynomials $L_n^{\alpha}$, 
the relevant bilinear generating function is the Hille-Hardy formula
$$
\sum_{n=0}^{\infty} \frac{n!}{\Gamma(n+\alpha+1)} L^{\alpha}_n(x)L^{\alpha}_n(y) \, r^n
	= \frac{1}{1-r} \exp\bigg( -\frac{(x+y)r}{1-r}\bigg) (xyr)^{-\alpha/2}
		I_{\alpha}\bigg( \frac{2\sqrt{xyr}}{1-r}\bigg), 
$$
where $|r|<1$, $\alpha>-1$ and $I_{\alpha}$ is the modified Bessel function of the first kind.
This identity was found in 1926 by Hille \cite{Hi} and independently rediscovered later by Hardy
\cite{Ha}, see \cite{Wa2}. An analogue of these formulas in the
Jacobi setting is Bailey's formula
\begin{align*}
& \sum_{n=0}^{\infty} \frac{P_n^{\ab}(\cos\theta) P_n^{\ab}(\cos\varphi)}{h_n^{\ab}}\, r^n =
	\frac{\Gamma(\alpha+\beta+2)}{2^{\alpha+\beta+1}\Gamma(\alpha+1)\Gamma(\beta+1)}\,
	\frac{1-r}{(1+r)^{\alpha+\beta+2}}\\
& \qquad \times F_4\Bigg( \frac{\alpha+\beta+2}2, \frac{\alpha+\beta+3}2; \alpha+1, \beta+1;
	\bigg( \frac{2\sin\frac{\theta}2 \sin\frac{\varphi}2}{r^{1/2}+r^{-1/2}}\bigg)^2,
	\bigg( \frac{2\cos\frac{\theta}2 \cos\frac{\varphi}2}{r^{1/2}+r^{-1/2}}\bigg)^2 \Bigg), 
\end{align*}
where $|r|<1$, $\alpha,\beta > -1$ and $F_4$ is Appel's hypergeometric function of two variables.
This generating function was first stated in 1935 without proof in Bailey's tract \cite{Ba}.
The proof is a straightforward consequence of Watson's formula for hypergeometric functions \cite{Wa1}
and was published slightly later \cite{Ba2}. 

However, in contrast with the Hermite and Laguerre cases,
Bailey's formula does not enable one to compute the Jacobi heat kernel. This is because the eigenvalues
$n(n+\alpha+\beta+1)$ occurring in the defining series are not linear in $n$.
It is known that in the four simple special cases $\ab = \pm 1/2$, 
the kernel $G_t^{\ab}(x,y)$ can be written by means of non-oscillating series.
The argument is based on the periodized Gauss-Weierstrass kernel and simple initial-value
problems for the classical heat equation in an interval.
No further elementary representation for $G_t^{\pm 1/2, \pm 1/2}(x,y)$ seems to be
possible. This indicates that there is little hope of deriving a closed formula for the Jacobi
heat kernel for general $\alpha$ and $\beta$
similar to those of Mehler, Hille and Hardy, and Bailey. A natural and desirable 
substitute for an exact expression are therefore the estimates in Theorem~\ref{thm:main}.

The Jacobi polynomials cover as special cases several other classical families of orthogonal polynomials,
including Chebyshev, Legendre and Gegenbauer (also called ultraspherical) polynomials. 
Special instances of the Jacobi heat kernel exist at least implicitly
in the literature since the 19th century, and the
question of describing its behavior was an open problem, even though perhaps never
stated explicitly in written form. Additional motivation comes from the fact that
$G_t^{\ab}(x,y)$ is also the transition probability density for the Jacobi diffusion process, 
which has important applications in stochastic modeling in physics, economics and genetics;
see \cite{KM} and references given there.
According to our knowledge, the behavior of $G_t^{\ab}(x,y)$ has not been investigated before, except for
its positivity. Bochner \cite{B} proved that the ultraspherical heat kernel is non-negative.
Strict positivity in the general Jacobi case was shown by Karlin and McGregor \cite{KM}.
Some later results on the positivity can be found in Gasper \cite{Ga} and Bochner \cite{B2}.

We also take the opportunity, see the Appendix, to describe the behavior of the Poisson-Jacobi
kernel, which is essentially the sum occurring in Bailey's formula. However, the representation
in terms of Appel's function $F_4$ does not seem to be very useful for this purpose.
Instead we employ a double integral representation that was derived recently by the authors \cite{NoSj2}
from a product formula due to Dijksma and Koornwinder \cite{DK}.

For short times $t$, a direct analytic treatment of the heavily oscillating series defining
$G_t^{\ab}(x,y)$ is practically impossible. 
Therefore, we develop a method combining several ingredients. These are, among others,
the already mentioned product formula of Dijksma and Koornwinder and a resulting
\emph{reduction formula}, transference of heat kernel estimates from a sphere,
a \emph{comparison principle} relating heat kernels for different type parameters,
the semigroup property and, finally, a rough estimate of the series defining $G_t^{\ab}(x,y)$.

The paper is organized as follows. In Section \ref{sec:prel} we introduce three related
Jacobi settings appearing in the literature, and explain how the associated heat kernels are connected.
Section \ref{sec:prep} contains the auxiliary results that form the main tools of our proof of
Theorem~\ref{thm:main}. Several of them are interesting in their own right.
Section~\ref{sec:proof} is devoted to the proof of Theorem~\ref{thm:main}. Maximal operators
of multi-dimensional Jacobi semigroups are treated in Section~\ref{sec:max}; by means of
Theorem \ref{thm:main}, we prove weak type $(1,1)$ estimates for these operators. 
Finally, in the Appendix we prove sharp estimates for the Poisson-Jacobi kernel.

Throughout the paper we use standard notation.
The letter $C$ will stand for many different positive constants independent of significant quantities. 
When writing estimates, we use the notation $X \lesssim Y$ to indicate that $X \le CY$. 
We write $X \simeq Y$ when simultaneously $X \lesssim Y$ and $Y \lesssim X$.
Tracing the proof of Theorem \ref{thm:main}, it is easy to verify that the constants in the
statement come out as claimed.

\vspace{5pt}

%%%%%%%%%%%%%%%%%%%%%%%%%%%%%%%%%%%%%%%%%%%%%%%%%%%%%%%%%%%%%%%%%%%%%%%%%%%%%%%%%%%%%%%%%%%%%%%%%%%%%
\section{Preliminaries} \label{sec:prel}
%%%%%%%%%%%%%%%%%%%%%%%%%%%%%%%%%%%%%%%%%%%%%%%%%%%%%%%%%%%%%%%%%%%%%%%%%%%%%%%%%%%%%%%%%%%%%%%%%%%%%

Given $\alpha,\beta > -1$, the one-dimensional Jacobi polynomials of type $\ab$ are defined
for $n \in \N$ and $-1<x<1$ by the Rodrigues formula (cf. \cite[(4.3.1)]{Sz})
$$
P_n^{\ab}(x) = \frac{(-1)^n}{2^n n!} (1-x)^{-\alpha}(1+x)^{-\beta} \frac{d^n}{dx^n}
	\Big[ (1-x)^{\alpha+n}(1+x)^{\beta+n}\Big].
$$
Note that each $P_n^{\ab}$ is a polynomial of degree $n$.

We will consider three closely related settings of orthogonal systems based on Jacobi polynomials.
All of them have deep roots in the existing literature.
Below we briefly introduce each setting, for the sake of simplicity in
dimension one. Multi-dimensional analogues arise in a standard way as tensor products
of one-dimensional systems.
%%%
\subsection*{Pure polynomial setting}
In this case the relevant system $\{P_n^{\ab} : n \ge 0\}$ is formed directly by Jacobi polynomials.
This system is an orthogonal basis in $L^2(d\r_{\ab})$, where $\r_{\ab}$ is the beta-type measure
given by
$$
d\r_{\ab}(x) = (1-x)^{\alpha} (1+x)^{\beta}\, dx
$$
in the interval $[-1,1]$. Each $P_n^{\ab}$ is an eigenfunction of the Jacobi differential operator
$$
J^{\ab} = -(1-x^2)\frac{d^2}{dx^2} - \big[ \beta-\alpha - (\alpha+\beta+2)x\big] \frac{d}{dx};
$$
more precisely
$$
J^{\ab} P_n^{\ab} = n(n+\alpha+\beta+1) P_n^{\ab}, \qquad n \ge 0.
$$
The operator $J^{\ab}$ is symmetric and nonnegative in $L^2(d\r_{\ab})$ on the domain
$C_c^2(-1,1)$, and has a natural self-adjoint extension whose spectral resolution is given by the
$P_n^{\ab}$; see \cite{NoSj} for details. The semigroup $T_t^{\ab} = \exp(-t J^{\ab})$
is a symmetric diffusion semigroup in the sense of \cite[Chapter 3]{topics};
in particular, $T_t^{\ab}\boldsymbol{1} = \boldsymbol{1}$.
It is also the transition semigroup for the Jacobi diffusion process, which already
received attention; cf.~\cite{KM} and references there. Some aspects of harmonic analysis in the
multi-dimensional Jacobi pure polynomial setting were investigated by the authors in \cite{NoSj}.

The integral representation of $T_t^{\ab}$, valid for $f\in L^1(d\r_{\ab})$, is
$$
T_t^{\ab}f(x) = \int_{-1}^1 G_t^{\ab}(x,y) f(y) \, d\r_{\ab}(y),
$$
where the Jacobi heat kernel is given by the oscillating series \eqref{ker_ser}.
The normalizing constants $h_n^{\ab} := \|P_n^{\ab}\|^2_{L^2(d\r_{\ab})}$ in \eqref{ker_ser}
are given by (cf. \cite[(4.3.3)]{Sz})
\begin{equation} \label{2.X}
h_n^{\ab} = \frac{2^{\alpha+\beta+1}\Gamma(n+\alpha+1)\Gamma(n+\beta+1)}
	{(2n+\alpha+\beta+1)\Gamma(n+\alpha+\beta+1)\Gamma(n+1)},
\end{equation}
where for $n=0$ and $\alpha+\beta=-1$ the product $(2n+\alpha+\beta+1)\Gamma(n+\alpha+\beta+1)$
must be replaced by $\Gamma(\alpha+\beta+2)$. Notice that $1/h_n^{\ab} \simeq n$, $n \ge 1$.
As already mentioned, $G_t^{\ab}(x,y)$ is strictly positive for $x,y \in [-1,1]$ and $t>0$.
Further, $G_t^{\ab}(x,y)$ is continuous (and even smooth) in $(t,x,y) \in (0,\infty)\times [-1,1]^2$, 
as well as in $(\ab) \in (-1,\infty)^2$. 
This can be verified by analyzing the defining series, using the corresponding continuity properties
of Jacobi polynomials and the bound 
$$
|P_n^{\ab}(x)| \lesssim (n+1)^{C}, \qquad n \ge 0, \quad x \in [-1,1],
$$
justified in Section \ref{sec:prep}. By \cite[Proposition 3.3]{NoSj}, 
$T_t^{\ab}f(x)$ is for any $f \in L^1(d\r_{\ab})$
a $C^{\infty}$ function of $(t,x)\in (0,\infty)\times (-1,1)$ satisfying the heat equation
$$
\big( \partial_t + J^{\ab}\,\big) T_t^{\ab}f(x) = 0, \qquad x \in (-1,1), \quad t>0.
$$
Finally, we note that (see \cite[p.\,347]{NoSj}) for $f \in C[-1,1]$ 
$$
\lim_{t \to 0^+} T_t^{\ab}f(x) = f(x), \qquad x \in [-1,1],
$$ 
and the convergence is uniform in $x$. All these facts will be used in the sequel.
%%%
\subsection*{Trigonometric polynomial setting}
This framework emerges if one applies the natural and convenient trigonometric parametrization
$x=\cos\theta$, $\theta \in [0,\pi]$, to the Jacobi polynomials. We consider the normalized trigonometric
polynomials
$$
\P_n^{\ab}(\theta) = 2^{(\alpha+\beta+1)/2} \big( h_n^{\ab}\big)^{-1/2} P_n^{\ab}(\cos\theta).
$$
The system $\{\P_n^{\ab} : n \ge 0\}$ is orthonormal and complete in $L^2(d\m_{\ab})$, where
$$
d\m_{\ab}(\theta) = \Big( \sin\frac{\theta}2 \Big)^{2\alpha+1} 
	\Big( \cos\frac{\theta}2\Big)^{2\beta+1} d\theta
$$
in $[0,\pi]$. Each $\P_n^{\ab}$ is an eigenfunction of the differential operator 
$$
\J^{\ab} = - \frac{d^2}{d\theta^2} - \frac{\alpha-\beta+(\alpha+\beta+1)\cos\theta}{\sin \theta}
	\frac{d}{d\theta} + \Big( \frac{\alpha+\beta+1}{2}\Big)^2;
$$
indeed
$$
\J^{\ab} \P_n^{\ab} = \Big( n + \frac{\alpha+\beta+1}{2}\Big)^2 \P_n^{\ab}, \qquad n \ge 0.
$$

The operator $\J^{\ab}$ has a natural self-adjoint extension whose spectral resolution is given by the
$\P_n^{\ab}$, see \cite{NoSj2} for details. The semigroup $\mathcal{T}_t^{\ab}=\exp(-t\J^{\ab})$
has the integral representation
$$
\mathcal{T}_t^{\ab}f(\theta) 
	= \int_0^{\pi} \mathcal{G}_t^{\ab}(\theta,\varphi)f(\varphi)\, d\m_{\ab}(\varphi),
$$
valid for $f \in L^1(d\m_{\ab})$, with the heat kernel defined by
$$
\mathcal{G}_t^{\ab}(\theta,\varphi) = \sum_{n=0}^{\infty} e^{-t\left(n+\frac{\alpha+\beta+1}{2}\right)^2}
	\P_n^{\ab}(\theta) \P_n^{\ab}(\varphi).
$$
Note that $\J^{\ab}$ is obtained by transforming $J^{\ab}$ according to the change of variable
$x = \cos\theta$ and introducing the zero order term. The latter modification leads to eigenvalues
which are squares, and therefore the oscillating series defining the one-dimensional 
Poisson-Jacobi kernel in this setting,
\begin{equation*}
\mathcal{H}_t^{\ab}(\theta,\varphi) = \sum_{n=0}^{\infty} e^{-t\left|n+\frac{\alpha+\beta+1}{2}\right|}
	\P_n^{\ab}(\theta)\P_n^{\ab}(\varphi),
\end{equation*}
can be represented in a more convenient way; in particular, Bailey's formula applies.
On the other hand, the semigroup $\mathcal{T}_t^{\ab}$ is submarkovian, but not Markovian in general.

Several fundamental harmonic analysis operators related to the Jacobi trigonometric polynomial setting
were studied recently by the authors \cite{NoSj2}. The ultraspherical case was widely investigated
from a slightly different perspective in the seminal paper of Muckenhoupt and Stein \cite{MuS}, which
in 1965 initiated the development in harmonic analysis known as 
\emph{harmonic analysis of orthogonal expansions}.
In both cases, the analysis was based on the one-dimensional Poisson-Jacobi kernel;
see \cite{NoSj2} for further facts and references.
%%%
\subsection*{Trigonometric `function' setting}
This context originates naturally in connection with transplantation problems for Jacobi expansions
(see for instance \cite{Mu1,CNS} and references there) and is derived from the previous setting
by modifying the Jacobi trigonometric polynomials so as to make the resulting system orthogonal with
respect to Lebesgue measure $d\theta$ in $[0,\pi]$. More precisely, we consider the functions
$$
\phi_n^{\ab}(\theta) = \Big( \sin\frac{\theta}2 \Big)^{\alpha+1/2} \Big( \cos\frac{\theta}2\Big)^{\beta+1/2}
	\P_n^{\ab}(\theta), \qquad n \ge 0.
$$
Then the system $\{\phi_n^{\ab} : n \ge 0\}$ is an orthonormal basis in $L^2(d\theta)$.
The associated differential operator is, cf. \cite[Section 4.24]{Sz},
$$
\mathbb{J}^{\ab} = -\frac{d^2}{d\theta^2} + \frac{(\alpha-1/2)(\alpha+1/2)}{4\sin^2\frac{\theta}2}
	+ \frac{(\beta-1/2)(\beta+1/2)}{4\cos^2\frac{\theta}2}
$$
and we have
$$
\mathbb{J}^{\ab} \phi_n^{\ab} = \Big( n + \frac{\alpha+\beta+1}{2}\Big)^2 \phi_n^{\ab}, \qquad n \ge 0.
$$

The semigroup $\mathbb{T}_t^{\ab} = \exp(-t\mathbb{J}^{\ab})$, generated by the natural
self-adjoint extension of $\mathbb{J}^{\ab}$, has the integral representation, 
valid for $f \in L^2(d\theta)$,
$$
\mathbb{T}_t^{\ab}f(\theta) = \int_0^{\pi} \mathbb{G}_t^{\ab}(\theta,\varphi)f(\varphi)\, d\varphi,
$$
where
$$
\mathbb{G}_t^{\ab}(\theta,\varphi) = \sum_{n=0}^{\infty} e^{-t\left(n+\frac{\alpha+\beta+1}{2}\right)^2}
	\phi_n^{\ab}(\theta) \phi_n^{\ab}(\varphi).
$$
Note that $\mathbb{T}_t^{\ab}$ is not defined on all $L^p(d\theta)$, $1<p<\infty$,
if $\alpha<-1/2$ or $\beta<-1/2$.

The Poisson-Jacobi kernel in the Jacobi trigonometric `function' setting is defined by
$$
\mathbb{H}_t^{\ab}(\theta,\varphi) = \sum_{n=0}^{\infty} e^{-t\left|n+\frac{\alpha+\beta+1}{2}\right|}
	\phi_n^{\ab}(\theta) \phi_n^{\ab}(\varphi). \vspace{10pt}
$$

Observe that there is a simple analytic connection between the heat kernels in the three Jacobi
frameworks. In fact, we have
\begin{align}
\mathbb{G}_t^{\ab}(\theta,\varphi) & = 
	\Big( \sin\frac{\theta}2 \sin\frac{\varphi}2 \Big)^{\alpha+1/2} 
	\Big( \cos\frac{\theta}2 \cos\frac{\varphi}2 \Big)^{\beta+1/2}
	\mathcal{G}_t^{\ab}(\theta,\varphi) \nonumber \\
& = 2^{\alpha+\beta+1} e^{-t\left(\frac{\alpha+\beta+1}2\right)^2}
	\Big( \sin\frac{\theta}2 \sin\frac{\varphi}2 \Big)^{\alpha+1/2} 
	\Big( \cos\frac{\theta}2 \cos\frac{\varphi}2 \Big)^{\beta+1/2}
	G_t^{\ab}(\cos\theta,\cos\varphi). \label{rel_ker}
\end{align}
Similarly, for the Jacobi-Poisson kernels,
$$
\mathbb{H}_t^{\ab}(\theta,\varphi) = 
	\Big( \sin\frac{\theta}2 \sin\frac{\varphi}2 \Big)^{\alpha+1/2} 
	\Big( \cos\frac{\theta}2 \cos\frac{\varphi}2 \Big)^{\beta+1/2}
	\mathcal{H}_t^{\ab}(\theta,\varphi).
$$
Thus kernel estimates can be translated between the Jacobi settings.

\vspace{5pt}

%%%%%%%%%%%%%%%%%%%%%%%%%%%%%%%%%%%%%%%%%%%%%%%%%%%%%%%%%%%%%%%%%%%%%%%%%%%%%%%%%%%%%%%%%%%%%%%%%%%%%
\section{Preparatory results} \label{sec:prep}
%%%%%%%%%%%%%%%%%%%%%%%%%%%%%%%%%%%%%%%%%%%%%%%%%%%%%%%%%%%%%%%%%%%%%%%%%%%%%%%%%%%%%%%%%%%%%%%%%%%%%

In this section, we prove several results that will be important ingredients of the proof of
Theorem \ref{thm:main}. They are of independent interest, and so some of them are stated
in slightly larger generality than actually needed for our present purposes.

It is convenient to introduce a compact notation for objects related to the ultraspherical setting, i.e.,
when the Jacobi parameters are equal, say $\alpha=\beta=\lambda$. In such cases, the sub- or superscript
$\lambda,\lambda$ will be shortened to $\lambda$; for instance
$$
P_n^{\lambda}:= P_n^{\lambda,\lambda}, \qquad \r_{\lambda} := \r_{\lambda,\lambda}, \qquad
	h_n^{\lambda} := h_n^{\lambda,\lambda}. 
$$
Notice that this convention differs somewhat from the standard notation for the classical ultraspherical
(Gegenbauer) polynomials $C_n^{\lambda}$, cf. \cite[Section 4.7]{Sz}. In fact we have, see
\cite[(4.7.1)]{Sz},
\begin{equation} \label{geg_rel}
C_n^{\lambda}(x) = \frac{\Gamma(\lambda+1/2)}{\Gamma(2\lambda)} \,
	\frac{\Gamma(n+2\lambda)}{\Gamma(n+\lambda+1/2)}\, P_n^{\lambda-1/2}(x), \qquad \lambda > -1/2,
	\quad \lambda \neq 0.
\end{equation}

\subsection*{Reduction formula}
The following product formula for Jacobi polynomials was derived by Dijksma and Koornwinder \cite{DK}:
\begin{align*}
& P_n^{\ab}(1-2s^2)P_n^{\ab}(1-2t^2) = \frac{\Gamma(\alpha+\beta+1)\Gamma(n+\alpha+1)\Gamma(n+\beta+1)}
	{\pi n! \Gamma(n+\alpha+\beta+1)\Gamma(\alpha+1\slash 2)\Gamma(\beta+1\slash 2)} \\
& \quad \times \int_{-1}^1\int_{-1}^1 C_{2n}^{\alpha+\beta+1}\big(ust+v\sqrt{1-s^2}\sqrt{1-t^2}\big)
	(1-u^2)^{\alpha-1\slash 2}(1-v^2)^{\beta-1\slash 2} du dv.
\end{align*}
This formula is valid for $\alpha,\beta > -1/2$. We shall write it in a more suitable form which,
by a limiting argument, will be valid for all $\alpha,\beta \ge -1/2$.

Let $\Pi_{\alpha}$ be the probability measure on the interval $[-1,1]$
defined for $\alpha>-1\slash 2$ by
$$
d\Pi_{\alpha}(u) = \frac{\Gamma(\alpha+1)}{\sqrt{\pi}\Gamma(\alpha+1\slash 2)}
	(1-u^2)^{\alpha-1\slash 2} du.
$$
In the limit case $\alpha=-1\slash 2$, we put
$$
\Pi_{-1\slash 2} = \frac{1}{2}(\delta_{-1}+\delta_{1}),
$$
where $\delta_{\pm 1}$ denotes a point mass at $\pm 1$.
Note that $\Pi_{-1\slash 2}$ is the weak limit of $\Pi_{\alpha}$ as $\alpha \to -1\slash 2$.
We now rewrite the above product formula with $s=\sin\frac{\theta}2$ and $t=\sin\frac{\varphi}2$,
using \eqref{geg_rel}, the relation between $P_n^{\ab}$ and $\P_n^{\ab}$, the fact that
(cf. \cite[(4.1.1)]{Sz})
$$
P_n^{\lambda}(1)=\frac{\Gamma(n+\lambda+1)}{\Gamma(n+1)\Gamma(\lambda+1)}
$$
and the expression \eqref{2.X} for $h_n^{\ab}$. After some computations, one finds that
\begin{align*}
& \P_n^{\ab}(\theta)\P_n^{\ab}(\varphi) \\
& \quad = \frac{\sqrt{\pi}\Gamma(\alpha+\beta+3/2)}{\Gamma(\alpha+1)\Gamma(\beta+1)} 
	\iint d\Pi_{\alpha}(u)d\Pi_{\beta}(v)\,
	\frac{P_{2n}^{\alpha+\beta+1/2}(u \sin\frac{\theta}2\sin\frac{\varphi}2 + v \cos\frac{\theta}2
		\cos\frac{\varphi}2) P_{2n}^{\alpha+\beta+1/2}(1)}{h_{2n}^{\alpha+\beta+1/2}}.
\end{align*}
This formula holds for all $\alpha,\beta \ge -1/2$, since Jacobi polynomials are continuous
functions of their type parameters (see \cite[(4.21.2)]{Sz}). 

Multiplying both sides above by $\exp(-t(n+\frac{\alpha+\beta+1}2)^2)$ 
and summing over $n \ge 0$ we get
\begin{align*}
\mathcal{G}_t^{\ab}(\theta,\varphi) & = 
\frac{\sqrt{\pi}\Gamma(\alpha+\beta+3/2)}{\Gamma(\alpha+1)\Gamma(\beta+1)}
	\, e^{-t(\alpha+\beta+1)^2/4} \iint d\Pi_{\alpha}(u)d\Pi_{\beta}(v)  \\
		& \qquad \times \sum_{n=0}^{\infty} e^{-t n(n+\alpha+\beta+1)}\,
		\frac{P_{2n}^{\alpha+\beta+1/2}(u \sin\frac{\theta}2\sin\frac{\varphi}2 + v \cos\frac{\theta}2
		\cos\frac{\varphi}2) P_{2n}^{\alpha+\beta+1/2}(1)}{h_{2n}^{\alpha+\beta+1/2}}.
\end{align*}
Writing $t n(n+\alpha+\beta+1)=\frac{t}4 2n [2n + (\alpha+\beta+1/2) + (\alpha+\beta+1/2) +1]$
and taking into account that ultraspherical polynomials of even (odd) orders are even (odd) functions
(cf. \cite[(4.7.4)]{Sz}), we see that the last series represents the even part with respect to the
first variable of the ultraspherical heat kernel with the second variable fixed at the endpoint $1$.
Since for symmetry reasons the corresponding odd part gives no contribution to the integral,
we end up with the following reduction formula.
\begin{thm} \label{thm:prod}
Let $\alpha,\beta \ge -1/2$. Then, for all $\theta,\varphi \in [0,\pi]$ and $t>0$,
$$
G_t^{\ab}(\cos\theta, \cos\varphi)
= \mathcal{C}_{\ab} \,
	\iint 
	G_{t/4}^{\alpha+\beta+1/2}\Big( u \sin\frac{\theta}2\sin\frac{\varphi}2 + v \cos\frac{\theta}2
		\cos\frac{\varphi}2, 1 \Big)\, d\Pi_{\alpha}(u)d\Pi_{\beta}(v),
$$
with the constant $\mathcal{C}_{\ab} 
= \sqrt{\pi}\Gamma(\alpha+\beta+3/2)/(2^{\alpha+\beta+1}\Gamma(\alpha+1)\Gamma(\beta+1))$.
\qed
\end{thm}

Thus we have expressed the general Jacobi heat kernel in terms of the ultraspherical one. 
The relation between the type parameters here will be essential for our arguments to prove
the heat kernel estimates.

\subsection*{Connection with the spherical heat kernel}
We now consider the Jacobi polynomial setting in the half-integer ultraspherical case, that is when
$$
\alpha=\beta = \frac{N}2-1, \qquad N=1,2,\ldots .
$$
It is well known that this situation is closely connected with expansions
in spherical harmonics on the Euclidean unit sphere of dimension $N$, see for instance
\cite[Section III]{KM}. In particular, there exists a relation between the heat kernels in the two
settings, which we indicate below.

For $N \ge 1$, let $S^N$ be the unit sphere in $\R^{N+1}$ and denote by $\sigma_{N}$ the standard
(non-normalized) area measure on $S^N$. 
The Laplace-Beltrami operator $\Delta_N$ on $S^N$ is symmetric and nonnegative in
$C^{\infty}(S^N)\subset L^2(d\sigma_N)$. The classical system of spherical harmonics on $S^N$
is an orthogonal basis in $L^2(d\sigma_N)$ of eigenfunctions of $\Delta_N$.
The spherical heat semigroup $U_t^N = \exp(t\Delta_N)$,
generated by the self-adjoint extension of $\Delta_N$, has an integral representation
$$
U_t^N f(\xi) = \int_{S^N} K_t^N(\xi,\eta) f(\eta)\, d\sigma_N(\eta), \qquad \xi \in S^N, \quad t>0,
$$
for $L^2(d\sigma_N)$.
The spherical heat kernel $K_t^N(\xi,\eta)$ can be expressed explicitly as an oscillatory series of
spherical harmonics, see \cite[p.\,176]{KM}.
By general theory (cf. \cite[Theorem 5.2.1]{Da}), $K_t^N(\xi,\eta)$ is a strictly positive and continuous
(even smooth) function of $(t,\xi,\eta) \in (0,\infty)\times S^N \times S^N$. 

It is well known that the zonal case in the context of $\Delta_N$ and
expansions with respect to spherical harmonics reduces to the ultraspherical setting 
in the interval $[-1,1]$ with the type parameter $\lambda = N/2-1$. 
Indeed, let $F$ be a zonal function on $S^N$, say $F = f \circ \psi$, where
$$
\psi(\xi) = \xi_1, \qquad \xi \in S^N,
$$
is the zonal projection onto the diameter of $S^N$ determined by the first coordinate axis. Then
the expansion of $F$ in spherical harmonics reduces to the expansion of $f$ in
ultraspherical polynomials $P_n^{\lambda}$ of type $\lambda=N/2-1$.
The associated heat semigroups are related in a similar way, as stated below.
\begin{lem} \label{lem:con}
Assume that $N \ge 1$ and $\lambda = N/2-1$. Then we have for $f \in L^2(d\r_{\lambda})$
$$
\big( T_t^{\lambda}f \big) \circ \psi(\xi) = U_t^N(f\circ \psi)(\xi), \qquad \xi \in S^N, \quad t>0.
$$
\end{lem}

\begin{proof}
Observe that the semigroups considered consist of $L^2$-bounded linear operators, which for each $t>0$
map $L^2$ into the subspaces of continuous functions on $[-1,1]$ and $S^N$, respectively. Moreover,
$f \in L^2(d\r_{\lambda})$ if and only if $f\circ \psi \in L^2(d\sigma_N)$. Therefore, since
linear combinations of ultraspherical polynomials are dense in $L^2(d\r_{\lambda})$, 
we may assume that $f = P_k^{\lambda}$ for some $k$. 

The identity we must prove has roots in the fact that the ultraspherical operator $J^{\lambda}$
is essentially the zonal part of the Laplace-Beltrami operator $\Delta_N$ on $S^N$.
Indeed, writing the differential operator $\Delta_N$ in hyperspherical coordinates on $S^N$
(see \cite[p.\,175]{KM}) one easily verifies that 
$$
\big( J^{\lambda} P_k^{\lambda}\big) \circ \psi = -\Delta_N \big( P_k^{\lambda}\circ \psi \big).
$$
Since $P_k^{\lambda}$ is an eigenfunction of $J^{\lambda}$, we see that $P_k^{\lambda}\circ \psi$
is an eigenfunction of $\Delta_N$, with the same eigenvalue. For smooth
functions, the self-adjoint extension of $\Delta_N$ coincides with the differential operator,
and we conclude that
$$
 \big( T_t^{\lambda} P_k^{\lambda} \big) \circ \psi(\xi) = 
	e^{-t k (k+2\lambda+1)} P_k^{\lambda}\circ \psi(\xi)
	= U_t^N \big(P_k^{\lambda} \circ \psi \big)(\xi), \qquad \xi \in S^N,
$$
as desired.
\end{proof}

We now establish a connection between the ultraspherical and spherical heat kernels.
\begin{thm} \label{thm:con}
Assume that $N \ge 1$ and $\lambda = N/2-1$. Then
$$
G_t^{\lambda}(x,y) 
= \int_{S^{N-1}} K_t^{N}\Big(\xi,\big(y,\zeta\sqrt{1-y^2}\big)\Big)\, d\sigma_{N-1}(\zeta),
	\qquad x,y \in [-1,1], \quad t>0,
$$
where $\xi \in \psi^{-1}(\{x\})$ is arbitrary.
\end{thm}

\begin{proof}
Let $f$ be a polynomial on $[-1,1]$. By Lemma \ref{lem:con}
$$
\int_{-1}^1 G_t^{\lambda}\big(\psi(\xi),y\big) f(y) (1-y^2)^{\lambda}\, dy =
\int_{S^N} K_t^{N}(\xi,\eta) f\circ \psi(\eta)\, d\sigma_N(\eta).
$$
To treat the last integral, we introduce zonal coordinates on $S^N$, 
$$
\Psi\colon [-1,1]\times S^{N-1} \longrightarrow S^N, \qquad \Psi(y,\zeta) = \big(y,\zeta\sqrt{1-y^2}\big).
$$
Then for reasonable functions $F$
$$
\int_{S^N} F(\xi)\, d\sigma_{N}(\xi) = \int_{-1}^1 \int_{S^{N-1}} F\circ \Psi(y,\zeta)\, 
	d\sigma_{N-1}(\zeta)\; (1-y^2)^{N/2-1}\, dy.
$$
Therefore,
$$
\int_{-1}^1 G_t^{\lambda}\big(\psi(\xi),y\big) f(y) (1-y^2)^{\lambda}\, dy =
\int_{-1}^1 \int_{S^{N-1}} K_t^{N}\Big(\xi,\big(y,\zeta\sqrt{1-y^2}\big)\Big)\, d\sigma_{N-1}(\zeta) \;
	f(y) (1-y^2)^{\lambda}\, dy.
$$
Since polynomials are dense in $C[-1,1]$ and the kernels in question are continuous functions of their
arguments, it follows that
$$
G_t^{\lambda}\big(\psi(\xi),y\big) 
= \int_{S^{N-1}} K_t^{N}\Big(\xi,\big(y,\zeta\sqrt{1-y^2}\big)\Big)\, d\sigma_{N-1}(\zeta),
$$
for $t>0$, $y \in [-1,1]$ and $\xi \in S^N$.
\end{proof}

The case $y=1$ and $\xi = (x,\sqrt{1-x^2},0,\ldots,0)$ of Theorem \ref{thm:con} reveals a
particularly simple relation between the ultraspherical and spherical heat kernels.
\begin{cor} \label{cor:con}
Let $\lambda = N/2 -1$ for some $N \in \{1,2,\ldots\}$. Then 
$$
G_t^{\lambda}(x,1) = \sigma_{N-1}(S^{N-1}) \; K_t^N\Big( \big(x,\sqrt{1-x^2},0,\ldots,0\big),
	(1,0,\ldots,0)\Big), \qquad x \in [-1,1], \quad t>0.
$$
\end{cor}
This expression for the ultraspherical heat kernel in terms of the spherical one 
will allow us to transfer qualitatively
sharp heat kernel bounds on spheres to the ultraspherical setting with half-integer type index.
On the other hand, it is interesting to observe that the spherical heat kernel on $S^N$ is completely
determined by $G_t^{\lambda}(x,1)$ for $\lambda = N/2-1$. This is a consequence of Corollary \ref{cor:con}
and the fact that $K_t^N(\xi,\eta)$ depends on $\xi$ and $\eta$ only through their spherical distance.

\subsection*{Comparison principle}
Given $\epsilon,\delta \ge 0$, define 
$$
\Phi_{\epsilon,\delta}(x) = (1-x)^{\epsilon\slash 2}(1+x)^{\delta\slash 2}, \qquad x \in [-1,1],
$$
with the convention that $(1\pm x)^0 = 1$ for $x = \mp 1$.
This is the square root of the Radon-Nikodym derivative $d\r_{\alpha+\epsilon,\beta+\delta}/ d\r_{\ab}$.
Using a parabolic PDE technique, we shall prove the following result comparing Jacobi heat kernels
with different type parameters.
\begin{thm} \label{thm:comp}
Let $\alpha,\beta > -1$.
Given $\epsilon, \delta \ge 0$ and $\alpha \ge -\epsilon \slash 2$, $\beta \ge -\delta \slash 2$,
we have
\begin{equation} \label{in_cmp}
	\Phi_{\epsilon,\delta}(x) \Phi_{\epsilon,\delta}(y)\,
		G_t^{\alpha+\epsilon,\beta+\delta}(x,y) 
	\le \exp\bigg({ \frac{\epsilon+\delta}{2} \Big(\alpha+\beta+1+\frac{\epsilon+\delta}{2}\Big)t}\bigg)
		\, G_t^{\ab}(x,y)
\end{equation}
for all $x,y \in [-1,1]$ and $t>0$.
\end{thm}
Translating this estimate to the other two Jacobi settings, we get
\begin{cor} \label{cor:comp}
Let $\alpha,\beta,\epsilon,\delta$ be as in Theorem \ref{thm:comp}. Then
$$
\Big( \sin\frac{\theta}2 \sin\frac{\varphi}2 \Big)^{\epsilon}
\Big( \cos\frac{\theta}2 \cos\frac{\varphi}2 \Big)^{\delta} 
\mathcal{G}_t^{\alpha+\epsilon,\beta+\delta}(\theta,\varphi) \le \mathcal{G}_t^{\ab}(\theta,\varphi)
$$
and
$$
\mathbb{G}_t^{\alpha+\epsilon,\beta+\delta}(\theta,\varphi)  \le \mathbb{G}_t^{\ab}(\theta,\varphi),
$$
for all $\theta,\varphi \in [0,\pi]$ and $t>0$, with the natural convention for boundary values in the
second inequality.
\end{cor}
The relation in the Jacobi trigonometric `function' setting is particularly nice since it shows that
the heat kernel is decreasing as a function of each of the type parameters $\alpha \ge 0$ and $\beta \ge 0$.
Note that, by subordination, the estimates of Corollary \ref{cor:comp} carry over to the corresponding
Poisson kernels.

\begin{proof}[Proof of Theorem \ref{thm:comp}]
Since the Jacobi heat kernel is a continuous function of its type parameters $\ab > -1$, we may
assume that $\epsilon, \delta > 0$. We first rewrite \eqref{in_cmp} in integrated form.
By integrating against $f(y)\, d\r_{\ab}(y)$, we see that \eqref{in_cmp} implies
\begin{equation} \label{p1p}
\Phi_{\epsilon,\delta}(x) T_t^{\alpha+\epsilon,\beta+\delta}( f/ \Phi_{\epsilon,\delta})(x) \le
 e^{ \frac{\epsilon+\delta}{2} (\alpha+\beta+1+\frac{\epsilon+\delta}{2})t}  \,  T_t^{\ab}f(x)
\end{equation}
for suitable functions $f\ge 0$. Conversely, if \eqref{p1p} holds for all nonnegative 
$f \in C^{\infty}_c(-1,1)$, then \eqref{in_cmp} will follow, since $G_t^{\ab}(x,y)$ is continuous
in $(x,y) \in [-1,1]^2$. We shall thus prove the lemma by verifying \eqref{p1p} for $x \in (-1,1)$
and $t>0$ and any $0 \le f \in C^{\infty}_c(-1,1)$ not identically $0$. Our reasoning will rely
on a generalization of the minimum principle method used to prove \cite[Lemma 3.4]{NoSj}.

Denote by $u = u(t,x)$ the left-hand side of \eqref{p1p} and let
$$
v=v(t,x)=e^{t\eta} 
e^{ \frac{\epsilon+\delta}{2} (\alpha+\beta+1+\frac{\epsilon+\delta}{2})t} \, T^{\ab}_t(f+\eta)(x)
$$
for some fixed $\eta >0.$ Since $f$ is smooth, the functions $u$ and $v$ have continuous
extensions to $[0,\infty) \times (-1,1)$. Our task will be done once we show that
$$
u(t,x) \le v(t,x), \qquad x \in (-1,1),
$$
for all $t \ge 0$ and any $\eta > 0$. Let
$$
T = \sup \big\{ t' \ge 0 : u(t,x) \le v(t,x) \; \textrm{for} \; (t,x) \in
[0,t') \times (-1,1) \big\}.
$$
Clearly, $u(0,x) = f(x) < f(x)+\eta = v(0,x)$ for $x \in (-1,1)$.
Moreover, $u(t,x) < v(t,x)$ for all $t\ge 0$ provided that $|x|$ is
sufficiently close to $1$; this is because $u(t,x) < C\Phi_{\epsilon,\delta}(x)$ and $v(t,x)\ge\eta$ for
$t \ge 0$, $x\in (-1,1)$.
Hence for $t$ small enough $u(t,x) < v(t,x)$, $x \in (-1,1)$, which means that $T>0$.

Suppose that $T$ is finite. We shall then derive a contradiction, which will end the reasoning.
Observe that $u(T,x) \le v(T,x)$ for all $x \in (-1,1)$ and $u(T,x_0) = v(T,x_0)$ for some $x_0 \in (-1,1)$.
We claim that
\begin{equation} \label{d_in}
\partial_t \big(v(t,x)-u(t,x)\big)\big|_{(t,x)=(T,x_0)} > 0.
\end{equation}
This would imply that $v(t,x_0)-u(t,x_0) < 0$ for $t$ slightly less than $T$, a contradiction.

To prove this claim, we compute the derivative in \eqref{d_in}. With the aid of the heat
equation we get
\begin{align*}
& \partial_t \big(v(t,x)-u(t,x)\big)  = \\ & \qquad
		 \bigg(\frac{\epsilon+\delta}{2}
	 \Big(\alpha+\beta+1+\frac{\epsilon+\delta}{2}\Big) + \eta \bigg)
    v(t,x) - {J}^{\ab} v(t,x) +
    \Phi_{\epsilon,\delta}(x) {J}^{\alpha+\epsilon,\beta+\delta} 
    	\big( {u(t,x)}/ {\Phi_{\epsilon,\delta}(x)} \big).
\end{align*}
Then using the definition of ${J}^{\ab}$ and the fact that $v-u = \partial_x(v-u)=0$ at 
the point $(T,x_0)$, we find after somewhat lengthy computations 
that the left-hand side in \eqref{d_in} is equal to
$$
(1-x_0^2)\partial_x^2(v-u)(T,x_0) +
     \bigg[ \frac{\epsilon (\alpha+\epsilon\slash 2)}{1-x_0} + 
     \frac{\delta(\beta+\delta\slash 2)}{1+x_0} \bigg] u(T,x_0)
    + \eta \, u(T,x_0).
$$
The first term above is nonnegative, since the function
$x \mapsto v(T,x)-u(T,x)$ has a local minimum at $x=x_0$.
The factor in the square bracket is obviously nonnegative by the assumptions
on $\epsilon,\delta,\alpha,\beta$. Finally, $u(T,x_0)$ is strictly positive by the 
corresponding property of the kernel involved. The claim follows.
\end{proof}

We remark that when either $\alpha < -\epsilon /2$ and $\epsilon > 0$ or $\beta < -\delta/2$ and
$\delta > 0$, the estimate of Theorem \ref{thm:comp} (and thus also the estimates of Corollary
\ref{cor:comp}) does not hold. This can be shown by means of a counterexample very similar to that
of \cite[Remark 3.6]{NoSj}.

\subsection*{Rough estimate}
We now employ absolute value estimates of Jacobi polynomials to obtain a rough short time bound for the
Jacobi heat kernel in terms of $t$ only. This method, of course, distinguishes no subtle effects
coming from oscillations. Therefore, the resulting estimate is far from sharp. 
More accurate upper bounds for the Jacobi heat kernel, involving also dependence on $x$ and $y$,
can be found by means of a more detailed analysis and the estimates for Jacobi polynomials
contained in \cite[Theorem 7.32.2]{Sz}; see also \cite[(7.32.6),(7.32.7)]{Sz}.

\begin{thm} \label{thm:abs}
Let $\alpha,\beta >-1$ and $T>0$ be fixed. Then
$$
G_t^{\ab}(x,y) \lesssim t^{-C_0}, \qquad x,y \in [-1,1], \quad 0< t \le T,
$$
where the constant $C_0$ depends only on $\alpha$ and $\beta$.
\end{thm}

In the proof we will use the following bound for Jacobi polynomials (see \cite[(7.32.2)]{Sz})
\begin{equation} \label{bg}
|P_n^{\ab}(x)| \lesssim n^{\gamma}, \qquad n \ge 1, \quad x \in [-1,1],
\end{equation}
where $\gamma=\max\{\alpha,\beta,-1/2\}$.

\begin{proof}[Proof of Theorem \ref{thm:abs}]
Recall that $1/h_n^{\ab} \simeq n$ for $n \ge 1$. Thus
$$
G_t^{\ab}(x,y) \lesssim 1 + \sum_{n=2}^{\infty} e^{-tn(n+\alpha+\beta+1)} n |P_n^{\ab}(x)| |P_n^{\ab}(y)|
	\le 1 + \sum_{n=2}^{\infty} e^{-tn} n |P_n^{\ab}(x)| |P_n^{\ab}(y)|,
$$
and \eqref{bg} implies
$$
G_t^{\ab}(x,y) \lesssim 1 + \sum_{n=2}^{\infty} e^{-tn} n^{2\gamma+1} \lesssim 
	1 + \frac{1}{t^{2\gamma+1}} \sum_{n=2}^{\infty} e^{-tn /2} \lesssim t^{-2\gamma-2}.
$$
The theorem is proved.
\end{proof}

\vspace{5pt}

%%%%%%%%%%%%%%%%%%%%%%%%%%%%%%%%%%%%%%%%%%%%%%%%%%%%%%%%%%%%%%%%%%%%%%%%%%%%%%%%%%%%%%%%%%%%%%%%%%%%%
\section{Proof of Theorem \ref{thm:main}} \label{sec:proof}
%%%%%%%%%%%%%%%%%%%%%%%%%%%%%%%%%%%%%%%%%%%%%%%%%%%%%%%%%%%%%%%%%%%%%%%%%%%%%%%%%%%%%%%%%%%%%%%%%%%%%

The large time behavior of $G_t^{\ab}(x,y)$ stated in Theorem \ref{thm:main} 
is a consequence of the short time estimate, in the following way.
Given $T>0$, the short time bound implies that
$$
G_T^{\ab}(z,y) \simeq 1, \qquad z,y \in [-1,1].
$$
By the semigroup property, for $t \ge T$ and $x,y \in [-1,1]$ one has
$$
G_t^{\ab}(x,y) = \int_{-1}^1 G_{t-T}^{\ab}(x,z) G_T^{\ab}(z,y)\, d\r_{\ab}(z).
$$
Since
$$
\int_{-1}^1 G_{t-T}^{\ab}(x,z)\, d\r_{\ab}(z) = T_{t-T}^{\ab}\boldsymbol{1}(x)=1, \qquad x \in [-1,1],
\quad t \ge T,
$$ 
we get the estimates
$$
G_t^{\ab}(x,y) \simeq 1, \qquad x,y \in [-1,1], \quad t \ge T.
$$
The existence of the uniform limit as $t \to \infty$ follows by combining the oscillating series
\eqref{ker_ser} with the estimate \eqref{bg} for Jacobi polynomials.

Thus it remains to prove the short time estimates, and we first introduce some further notation.
A real number $r$ will be called \emph{dyadic} if $r=n/2^k$ for some integers $n,k$.
We will use the notation $X \simeq \simeq Y \exp(-c Z)$ to indicate that 
$Y \exp(-c_1 Z) \lesssim X \lesssim Y \exp(-c_2 Z)$, with positive constants $c_1$ and $c_2$ independent of
significant quantities.
Thus the short time estimates of Theorem \ref{thm:main} can be written
\begin{align} \label{short_est}
& G_t^{\ab}(\cos\theta,\cos\varphi) \\
& \quad  \simeq \simeq \Big( t + \sin\frac{\theta}2\sin\frac{\varphi}2\Big)^{-\alpha-1/2}
	\Big( t + \cos\frac{\theta}2 \cos\frac{\varphi}2\Big)^{-\beta-1/2} \frac{1}{\sqrt{t}}
	\exp\bigg( - c \frac{(\theta-\varphi)^2}{t}\bigg). \nonumber
\end{align}
The quantity $T>0$ will be fixed for the rest of the proof of Theorem \ref{thm:main}.
For the sake of clarity, we divide the proof into several steps, as follows. 
\begin{itemize}
\item[\textbf{1.}] Estimate $G_t^{\lambda}(x,1)$ for half-integer $\lambda \ge -1/2$
	by transferring, via Theorem \ref{thm:con}, known bounds for the spherical heat kernel.
\item[\textbf{2.}] Starting from the bounds of Step 1,
	iterate the reduction formula (Theorem \ref{thm:prod}) to estimate
	$G_t^{\lambda}(x,1)$ for all dyadic values of $\lambda \ge -1/2$.
\item[\textbf{3.}] Apply the reduction formula to the estimate of Step 2
	to prove \eqref{short_est} when $\alpha,\beta \ge -1/2$
	and the sum $\alpha + \beta$ is a dyadic number.
\item[\textbf{4.}] Combine the estimate of Step 3 with the comparison principle 
	(Theorem \ref{thm:comp}) to obtain a weakened version of \eqref{short_est} in the ultraspherical
	case for $\alpha=\beta=\lambda > 0$.
\item[\textbf{5.}] Use the semigroup property and the rough estimate of Theorem \ref{thm:abs}
	to eliminate the weakening in Step 4 and prove \eqref{short_est} for $\alpha=\beta=\lambda > 0$.
\item[\textbf{6.}] Use the estimate of Step 5 and
	iterate with the reduction formula to estimate $G_t^{\lambda}(x,1)$
	for $-1/2 \le \lambda < 0$. In the final stroke 
	apply again the reduction formula to prove \eqref{short_est} for all $\alpha,\beta \ge -1/2$.
\end{itemize}

\subsection*{Step 1} We first invoke the well-known Gaussian bounds for the spherical heat kernel,
see \cite[Theorems 5.5.6 and 5.6.1]{Da}.
Let $d(\xi,\eta)=\arccos \langle \xi, \eta \rangle$ be the spherical distance between 
$\xi$ and $\eta \in S^N$. Given any $\delta > 0$, we have
$$
K_t^N(\xi,\eta) \lesssim \frac{1}{t^{N/2}} \exp\bigg( -\frac{d(\xi,\eta)^2}{4(1+\delta)t}\bigg),
	\qquad \xi,\eta \in S^N, \quad 0 < t \le T,
$$
and
$$
K_t^N(\xi,\eta) \ge \frac{1}{(4\pi t)^{N/2}} \exp\bigg( -\frac{d(\xi,\eta)^2}{4t}\bigg),
	\qquad \xi,\eta \in S^N, \quad t>0.
$$
These estimates, together with Corollary \ref{cor:con} and the observation that 
$d(\xi,\eta)\simeq |\xi-\eta|$ for $\xi,\eta \in S^N$, lead to the following result.
\begin{lem} \label{lem:1}
Assume that $\lambda = N/2 -1$ for some $N \ge 1$. Then
$$
G_t^{\lambda}(x,1) \simeq \simeq \frac{1}{t^{\lambda+1}} \exp\bigg(-c \frac{1-x}{t}\bigg), \qquad
	x \in [-1,1], \quad 0< t \le T,
$$
or equivalently,
\begin{equation} \label{est1}
G_t^{\lambda}(\cos\theta,1) \simeq \simeq \frac{1}{t^{\lambda+1}} \exp\bigg(- c \frac{\theta^2}{t}\bigg),
	\qquad \theta \in [0,\pi], \quad 0 < t \le T.
\end{equation}
\end{lem}
Notice that this estimate is a special case of \eqref{short_est}.

\subsection*{Step 2.} We claim that the bounds of Lemma \ref{lem:1} hold for all dyadic values of
$\lambda \ge -1/2$. To verify this, it is enough to prove the following lemma, since one can then iterate. 
\begin{lem} \label{lem:iter}
Assume that the estimate \eqref{est1} holds for some $\lambda > -1/2$.
Then it holds also with $\lambda$ replaced by $\lambda' = \lambda/2-1/4$.
\end{lem}

To prove this lemma, we need an auxiliary result.
\begin{lem} \label{lem:bes}
Let $\nu > -1/2$. Then
$$
\int \exp(zs) \, d\Pi_{\nu}(s) \simeq (1+z)^{-\nu-1/2}\, e^z, \qquad z \ge 0.
$$
\end{lem}

\begin{proof}
One can assume that $z>1$, since the opposite case is trivial. We split the integral and observe
that the integral taken over $(0,1)$ is larger than that over $(-1,0)$. Thus we need only consider
$$
\int_{0}^1 e^{zs}\, d\Pi_{\nu}(s) \simeq \int_0^1 e^{zs} (1-s)^{\nu-1/2}\, ds.
$$
Here we make the two transformations $t=1-s$ and $r=zt$, and get
$$
e^z z^{-\nu-1/2} \int_0^z e^{-r} r^{\nu-1/2}\, dr.
$$
The conclusion follows.
\end{proof}

\begin{proof}[{Proof of Lemma \ref{lem:iter}}]
Suppose that \eqref{est1} holds for some $\lambda > -1/2$. By Theorem \ref{thm:prod},
\begin{align}
G_t^{\lambda'}(\cos\theta,1) & = C\, \Pi_{\lambda'}([-1,1]) \int
	G_{t/4}^{\lambda}\Big( v \cos\frac{\theta}2, 1\Big) \, d\Pi_{\lambda'}(v) \nonumber \\
& \simeq \simeq \frac{1}{t^{\lambda+1}}
	\int
	\exp\bigg( -c\frac{1-v\cos\frac{\theta}2}{t} \bigg) \, d\Pi_{\lambda'}(v). \label{est2}
\end{align}
Applying now Lemma \ref{lem:bes} to the last integral in \eqref{est2},
or rather to the upper and lower estimates that \eqref{est2} stands for, we see that
\begin{align*}
G_t^{\lambda'}(\cos\theta,1) & \simeq \simeq \bigg( 1 + \frac{\cos\frac{\theta}2}{t}\bigg)^{-\lambda'-1/2}
	\frac{1}{t^{\lambda+1}} \exp\bigg( - c\frac{1-\cos\frac{\theta}2}{t}\bigg) \\
& \simeq \simeq \bigg( t + \cos\frac{\theta}2\bigg)^{-\lambda'-1/2} \frac{1}{t^{\lambda'+1}}
	\exp\bigg( - c \frac{\theta^2}{t}\bigg). 
\end{align*}

The last expression is what we need except for the first factor. However, since we consider
$\lambda' > -1/2$,
$$
\bigg( t + \cos\frac{\theta}2\bigg)^{-\lambda'-1/2} \gtrsim 1, \qquad \theta \in [0,\pi], \quad 0<t\le T,
$$
and, on the other hand,
$$
\bigg( t + \cos\frac{\theta}2\bigg)^{-\lambda'-1/2} \lesssim
\exp\bigg( \epsilon \frac{\theta^2}{t}\bigg), 
	\qquad \theta \in [0,\pi], \qquad t>0,
$$
for any $\epsilon > 0$,
as can easily be seen by considering separately the cases $\theta \le \pi /2$ and $\theta > \pi /2$.
So this factor is insignificant, and Lemma \ref{lem:iter} follows.
\end{proof}

\subsection*{Step 3} Let $\alpha,\beta \ge -1/2$ be such that $\alpha+\beta$ is a dyadic number.
Then $\lambda = \alpha+\beta+1/2 \ge -1/2$ is also a dyadic number, so in view of Theorem \ref{thm:prod}
and Step 2 we may write
\begin{align*}
& G_t^{\ab}(\cos\theta,\cos\varphi)\\ & \qquad = \mathcal{C}_{\alpha,\beta} \iint
	G_{t/4}^{\lambda}\Big( u \sin\frac{\theta}2\sin\frac{\varphi}2 + v\cos\frac{\theta}2\cos\frac{\varphi}2,
		1\Big) \, d\Pi_{\alpha}(u) d\Pi_{\beta}(v) \\
& \qquad \simeq \simeq	\frac{1}{t^{\lambda+1}}
	\int d\Pi_{\alpha}(u) \int \exp\bigg( 
		-c \frac{1-u\sin\frac{\theta}2\sin\frac{\varphi}2 -v\cos\frac{\theta}2\cos\frac{\varphi}2}{t} 
		\bigg) \, d\Pi_{\beta}(v).
\end{align*} 
Applying now Lemma \ref{lem:bes} as in Step 2, first to the integral in $v$ and then to that in $u$,
and observing that
$$
1 - \sin\frac{\theta}2\sin\frac{\varphi}2 - \cos\frac{\theta}2\cos\frac{\varphi}2
	= 2\sin^2\frac{\theta-\varphi}4 \simeq (\theta-\varphi)^2, \qquad \theta,\varphi \in [0,\pi],
$$
we get
\begin{align*}
& G_t^{\ab}(\cos\theta,\cos\varphi) \\ &\qquad  \simeq \simeq \frac{1}{t^{\lambda+1}}
	\bigg( 1 + \frac{\sin\frac{\theta}2\sin\frac{\varphi}2}t\bigg)^{-\alpha-1/2}
	\bigg( 1 + \frac{\cos\frac{\theta}2\cos\frac{\varphi}2}t\bigg)^{-\beta-1/2}
	\exp\bigg( - c\frac{\sin^2\frac{\theta-\varphi}4}{t}\bigg).
\end{align*}
From this, \eqref{short_est} follows.
In the next steps, we will remove the restriction that $\alpha+\beta$ is dyadic.

\subsection*{Step 4} Suppose that $\lambda > 0$ is arbitrary. Then there exist 
$\epsilon, \epsilon' >0$ such that $2\lambda -\epsilon$ and $2\lambda + \epsilon'$
are dyadic numbers and $\lambda- \epsilon > 0$. 
Applying Theorem \ref{thm:comp}~twice, with $\delta = 0$, $\beta = \lambda$
and either $\alpha = \lambda-\epsilon$ or $\alpha=\lambda$, we obtain
\begin{align*}
\Phi_{\epsilon,0}(x) \Phi_{\epsilon,0}(y) G_t^{\lambda}(x,y) & \le 
	e^{\frac{\epsilon}2(2\lambda+1-\frac{\epsilon}2)t} G_t^{\lambda-\epsilon,\lambda}(x,y),\\
\Phi_{\epsilon',0}(x) \Phi_{\epsilon',0}(y) G_t^{\lambda+\epsilon',\lambda}(x,y) & \le
	e^{\frac{\epsilon'}2(2\lambda+1+\frac{\epsilon'}2)t} G_t^{\lambda}(x,y),
\end{align*}
for all $x,y \in [-1,1]$ and $t>0$. This implies
$$
\Big(\sin\frac{\theta}2\sin\frac{\varphi}2\Big)^{\epsilon'} 
	G_t^{\lambda+\epsilon',\lambda}(\cos\theta,\cos\varphi) \lesssim G_t^{\lambda}(\cos\theta,\cos\varphi)
	\lesssim \Big(\sin\frac{\theta}2\sin\frac{\varphi}2\Big)^{-\epsilon}
	G_t^{\lambda-\epsilon,\lambda}(\cos\theta,\cos\varphi),
$$
uniformly in $\theta,\varphi \in [0,\pi]$ and $0<t\le T$; 
here and later on endpoint values are understood in a limiting sense, if necessary, and may be infinite.
In an analogous way, we may vary the second type parameter and use Theorem \ref{thm:comp} to get
$$
\Big(\cos\frac{\theta}2\cos\frac{\varphi}2\Big)^{\epsilon'} 
	G_t^{\lambda,\lambda+\epsilon'}(\cos\theta,\cos\varphi) \lesssim G_t^{\lambda}(\cos\theta,\cos\varphi)
	\lesssim \Big(\cos\frac{\theta}2\cos\frac{\varphi}2\Big)^{-\epsilon}
	G_t^{\lambda,\lambda-\epsilon}(\cos\theta,\cos\varphi),
$$
uniformly in $\theta,\varphi \in [0,\pi]$ and $0<t\le T$. 

Next we combine these estimates with those obtained in Step 3.
For some positive constants $c_1$ and $c_2$ and all
$\theta,\varphi \in [0,\pi]$ and $0 < t \le T$, this leads to
\begin{align} 
& \; \frac{1}{F_t(\theta,\varphi)^{\epsilon'}}
	\Big(t+\sin\frac{\theta}2\sin\frac{\varphi}2\Big)^{-\lambda-1/2}
	\Big(t+\cos\frac{\theta}2\cos\frac{\varphi}2\Big)^{-\lambda-1/2}
	\frac{1}{\sqrt{t}} \exp\bigg( - c_1 \frac{(\theta-\varphi)^2}{t}\bigg) \nonumber \\
\lesssim & \; G_t^{\lambda}(\cos\theta,\cos\varphi) \label{4.YZ} \\ 
\lesssim & \; F_t(\theta,\varphi)^{\epsilon}
	\Big(t+\sin\frac{\theta}2\sin\frac{\varphi}2\Big)^{-\lambda-1/2}
	\Big(t+\cos\frac{\theta}2\cos\frac{\varphi}2\Big)^{-\lambda-1/2}
	\frac{1}{\sqrt{t}} \exp\bigg( - c_2 \frac{(\theta-\varphi)^2}{t}\bigg), \nonumber
\end{align}
where the auxiliary function $F_t$ is defined by
$$
F_t(\theta,\varphi)  = \min\Bigg( 1+
	\frac{t}{\sin\frac{\theta}2\sin\frac{\varphi}2} \;,\;
	1+ \frac{t}{\cos\frac{\theta}2\cos\frac{\varphi}2}\Bigg).
$$
It is easy to verify that
$$
F_t(\theta,\varphi) \simeq 1+ \frac{t}{\cos\frac{\theta-\varphi}2}.
$$
Notice that the bounds in \eqref{4.YZ} coincide with those of \eqref{short_est}, except for the factors
involving $F_t$. In the next step, we will show how to deal with these factors.

\subsection*{Step 5} We shall see how \eqref{4.YZ} implies \eqref{short_est} with $\alpha=\beta=\lambda>0$.
Clearly, the factors $F_t^{\epsilon}$ and $1/F_t^{\epsilon'}$ in \eqref{4.YZ} are of importance only when 
$\theta$ and $\varphi$ are close to opposite endpoints of the interval $[0,\pi]$. Indeed, we have
$$
F_t(\theta,\varphi)^{\epsilon'} \simeq 1 \simeq F_t(\theta,\varphi)^{\epsilon}, \qquad |\theta-\varphi| \le 
	\frac{2\pi}3, \quad 0 < t \le T.
$$
Thus from now on we may assume that $|\theta-\varphi| > 2\pi/3$. Moreover, for symmetry reasons 
it is enough to consider the case $\theta < \pi/3$ and $\varphi > 2\pi/3$.

Observe that under these assumptions, in the right-hand side of \eqref{short_est}
only the exponential factor is significant, since it behaves like $\exp(-c/t)$ and
the other factors are essentially contained between $1$ and some negative power of $t$.
So what we must prove is simply that
\begin{equation} \label{p1}
G_t^{\lambda}(\cos\theta,\cos\varphi) \simeq \simeq \exp\bigg( -\frac{c}t\bigg)
\end{equation}
for $0 < t \le T$. 
We first verify \eqref{p1} under the additional assumption that 
$\theta \ge e^{-c_0/t}$ or $\varphi \le \pi-e^{-c_0/t}$, where $c_0>0$ is
a sufficiently small constant. Clearly, either
of these conditions implies $\cos\frac{\theta-\varphi}2 \gtrsim e^{-c_0/t}$. Then for $0 < t \le T$,
$$
F_t(\theta,\varphi)^{\epsilon} \simeq \bigg( 1 + \frac{t}{\cos\frac{\theta-\varphi}2}\bigg)^{\epsilon}
	\lesssim \big( 1 + T e^{c_0/t}\big)^{\epsilon} \lesssim e^{c_0\epsilon /t},
$$
and analogous bounds hold for $F_t(\theta,\varphi)^{\epsilon'}$.
This implies that if $c_0$ is taken small enough, the factors $F_t^{\epsilon}$ and $1/F_t^{\epsilon'}$
will be insignificant in \eqref{4.YZ}, and \eqref{p1} follows.

Thus we proved the following.
\begin{lem} \label{lem:p02}
Let $\lambda > 0$. There exists a constant $c_0>0$ such that $\eqref{short_est}$ 
with $\alpha=\beta=\lambda >0$ holds for all $\theta, \varphi \in [0,\pi]$ and $0 < t \le T$,
except possibly when $\theta \le e^{-c_0/t}$ and $\varphi \ge \pi - e^{-c_0/t}$ or vice versa.
\qed
\end{lem}

Finally, suppose that $\theta < e^{-c_0/t}$ and $\varphi > \pi-e^{-c_0/t}$. By the semigroup property,
$$
G_{2t}^{\lambda}(\cos\theta,\cos\varphi) = \int_0^{\pi} G_t^{\lambda}(\cos\theta,\cos\psi)
	G_t^{\lambda}(\cos\psi,\cos\varphi) (\sin\psi)^{2\lambda+1}\, d\psi.
$$
We split the interval of integration here into $D_1 = (0,e^{-c_0/t})$,
$D_2 = (e^{-c_0/t}, \pi-e^{-c_0/t})$ and $D_3 = (\pi-e^{-c_0/t},\pi)$, and we may
assume that $c_0$ is small enough so that $D_2 \supset [\pi/3,2\pi/3]$.
Denote the resulting integrals by $\mathcal{I}_1$, $\mathcal{I}_2$ and $\mathcal{I}_3$, respectively.
To estimate $\mathcal{I}_2$, we apply Lemma \ref{lem:p02}, getting
\begin{align*}
\mathcal{I}_2 \simeq\simeq & \int_{e^{-c_0/t}}^{\pi-e^{-c_0/t}}
	\bigg[ \Big( t+ \sin\frac{\theta}2\sin\frac{\psi}2\Big) \Big( t + \cos\frac{\theta}2\cos\frac{\psi}2\Big)
	\Big(t + \sin\frac{\psi}2\sin\frac{\varphi}2\Big) \Big( t + \cos\frac{\psi}2\cos\frac{\varphi}2\Big) 
	\bigg]^{-\lambda-1/2}\\ &\qquad \times \frac{1}{t} \exp\bigg( -c\frac{(\theta-\psi)^2 + 
	(\psi-\varphi)^2}{t}\bigg) (\sin\psi)^{2\lambda+1}\, d\psi.
\end{align*}
Here $(\theta-\psi)^2+(\psi-\varphi)^2 \simeq 1$, which means that the exponential in the integrand
makes all the other factors insignificant. Thus
$$
\mathcal{I}_2 \simeq \simeq \exp\bigg( -\frac{c}t\bigg).
$$
To bound $\mathcal{I}_1$, we apply the rough estimate of Theorem \ref{thm:abs}
to the two kernels in the integrand. We have
$$
\mathcal{I}_1 \lesssim \int_0^{e^{-c_0/t}} t^{-C} (\sin\psi)^{2\lambda+1}\, d\psi \lesssim
	\exp\bigg(-\frac{c_0}{2t}\bigg).
$$
Since the case of $\mathcal{I}_3$ is analogous, it follows that 
$\mathcal{I}_1+\mathcal{I}_2+\mathcal{I}_3$ and thus also $G_t^{\lambda}(\cos\theta,\cos\varphi)$
satisfy the estimate \eqref{p1}. 

Altogether, we have proved the following.
\begin{lem} \label{lem:p03}
The estimates \eqref{short_est} hold for $\alpha=\beta=\lambda > 0$
and all $\theta,\varphi \in [0,\pi]$ and $0 < t \le T$. \\
\qed
\end{lem}

\subsection*{Step 6} From Lemma \ref{lem:p03} with $\varphi=0$, we get as in the last part of Step 2
\begin{equation} \label{est7}
G_t^{\lambda}(\cos\theta,1) \simeq \simeq \frac{1}{t^{\lambda+1}} \exp\bigg(-c\frac{\theta^2}{t}\bigg),
 \qquad \theta \in [0,\pi], \quad 0< t \le T,
\end{equation}
provided that $\lambda \ge 0$; the case $\lambda=0$ is covered by Step 1.
Applying repeatedly Lemma \ref{lem:iter} as in Step 2, we conclude that \eqref{est7} holds for each 
$\lambda \ge -1/2$, since the case $\lambda = -1/2$ follows again from Step 1.

Finally, we combine \eqref{est7} for arbitrary $\lambda = \alpha+\beta + 1/2 \ge -1/2$
with the reduction formula, as done in Step 3 for dyadic $\lambda$. This establishes
the short time bound of Theorem \ref{thm:main} for general $\alpha,\beta \ge -1/2$.

The proof of Theorem \ref{thm:main} is complete.

\vspace{5pt}

%%%%%%%%%%%%%%%%%%%%%%%%%%%%%%%%%%%%%%%%%%%%%%%%%%%%%%%%%%%%%%%%%%%%%%%%%%%%%%%%%%%%%%%%%%%%%%%%%%%%%
\section{The heat maximal operators} \label{sec:max}
%%%%%%%%%%%%%%%%%%%%%%%%%%%%%%%%%%%%%%%%%%%%%%%%%%%%%%%%%%%%%%%%%%%%%%%%%%%%%%%%%%%%%%%%%%%%%%%%%%%%%

We shall now consider $d$-dimensional Jacobi settings. Each of the
 semigroups $T_t^{\ab}$,\hskip5pt $\mathcal{T}_t^{\ab}$ and  $\mathbb{T}_t^{\ab}$
has a natural $d$-dimensional  extension, see \cite[Section 2]{NoSj}.
In particular, their kernels are simply tensor  products of the corresponding one-dimensional heat kernels.
Letting now  $\alpha, \beta \in (-1,\infty)^d$ denote type multi-parameters, 
we can use the same notation as before for the semigroups, their kernels
and other related notions. The corresponding measure spaces will then be $([-1,1]^d,d\r_{\ab})$,
$([0,\pi]^d,d\m_{\ab})$ and $([0,\pi]^d,d\theta)$, respectively, where 
$$
d\r_{\ab} = \bigotimes_{i=1}^d d\r_{\alpha_i,\beta_i}, \qquad 
d\m_{\ab} = \bigotimes_{i=1}^d d\m_{\alpha_i,\beta_i},
$$
and $d\theta$ is  the $d$-dimensional Lebesgue measure in $[0,\pi]^d$.

This allows us to introduce multi-dimensional  maximal operators
$$
T^{\ab}_* f(x) = \sup_{t>0} \big|T_t^{\ab}f(x)\big|
$$
and 
$\mathcal{T}^{\ab}_*$ and $\mathbb{T}^{\ab}_*$ with analogous definitions.
Using
 Theorem \ref{thm:main}, we shall  show that these operators
satisfy weak type $(1,1)$ estimates in the corresponding measure spaces.
\begin{thm} \label{thm:max}
Let $d \ge 1$ and assume that $\alpha,\beta \in [-1/2,\infty)^d$. Then
\begin{itemize}
\item[(i)] $T_*^{\ab}$ is bounded from $L^1(d\r_{\ab})$ to weak $L^1(d\r_{\ab})$;
\item[(ii)] $\mathcal{T}_*^{\ab}$ is bounded from $L^1(d\m_{\ab})$ to weak $L^1(d\m_{\ab})$;
\item[(iii)] $\mathbb{T}_*^{\ab}$ is bounded from $L^1(d\theta)$ to weak $L^1(d\theta)$.
\end{itemize}
\end{thm}

An important consequence of Theorem \ref{thm:max} is the almost everywhere boundary convergence
for the Jacobi semigroups applied to $L^1$ functions. 
Note that by the subordination principle, Theorem \ref{thm:max} implies weak type $(1,1)$ estimates,
and thus also convergence results, for the multi-dimensional Poisson-Jacobi semigroups.

We briefly discuss $L^p$ bounds for these operators.
For $\alpha,\beta \in (-1,\infty)^d$, the boundedness of $T_*^{\ab}$ in $L^p(d\r_{\ab})$, $1<p\le \infty$,
follows by Stein's general maximal theorem \cite[Chapter 3]{topics};
see \cite[p.\,346]{NoSj}. In the restricted range  
$\ab \in [-1/2,\infty)^d$, it can also be obtained by interpolation between Theorem
\ref{thm:max}(i) and the trivial boundedness in $L^{\infty}$.
The case of $\mathcal{T}_*^{\ab}$ is much the same. In fact, $\mathcal{T}_*^{\ab}$ is  controlled
by $T^{\ab}_*$, as can be seen from the proof given below, and so it inherits the $L^p$ 
mapping properties of $T^{\ab}_*$. The $L^p$-boundedness of $\mathbb{T}_*^{\ab}$ is even simpler.
Indeed, as pointed out in the proof below, Theorem \ref{thm:main}
implies that when $\ab \in [-1/2,\infty)^d$, the operator  $\mathbb{T}_*^{\ab}$ is controlled by 
the standard maximal function in $[0,\pi]^d$ and hence $L^p$-bounded for $p>1$.
When $\alpha$ or $\beta$ are  not both in $[-1/2,\infty)^d$, the estimates of Theorem \ref{thm:main} can be
expected to hold in the same form, and this suggests that the behavior of $\mathbb{T}_*^{\ab}$ 
admits a similar anomaly to that occurring in certain Laguerre function settings and called
\emph{pencil phenomenon} \cite{NoSj3}.

\begin{proof}[Proof of Theorem \ref{thm:max}]
From Theorem \ref{thm:main} it follows that for large $t$  the three Jacobi heat kernels are bounded.
Thus, from now on, we may consider only the maximal operators defined
by taking suprema in the restricted range $0<t\le 1$.

We first treat (iii) and (i). In view of Theorem \ref{thm:main} and \eqref{rel_ker}, 
$$
\mathbb{G}_t^{\ab}(\theta,\varphi) \lesssim \frac{1}{\sqrt{t}}\exp\bigg( -c\frac{(\theta-\varphi)^2}{t}
	\bigg), \qquad \theta,\varphi \in [0,\pi], \quad 0 < t \le 1,
$$
where $c>0$ depends only on $\alpha$ and $\beta$. The right-hand side here is essentially the standard
Gaussian kernel, and so $\mathbb{T}_*^{\ab}$ can be controlled by the Hardy-Littlewood maximal
operator restricted to $[0,\pi]^d$. Therefore, $\mathbb{T}_*^{\ab}$ is of weak type $(1,1)$.

Next, we show that (i) follows from (ii). Observe that for $f \in L^1(d\r_{\ab})$ we have
$$
T_t^{\ab}f(\cos\theta)
 = e^{t\sum_{i=1}^d \left(\frac{\alpha_i+\beta_i+1}2\right)^2}
\mathcal{T}_t^{\ab}(f \circ \cos)(\theta), \qquad \theta \in [0,\pi],
$$
and consequently $|T_t^{\ab}f(\cos\theta)| \simeq  |\mathcal{T}_t^{\ab}(f\circ\cos)(\theta)| $
for $0 < t \le 1$. Thus (ii) implies (i). 

We pass to proving (ii). 
The kernel to be considered is then 
$ \mathcal{G}_t^{\alpha,\beta} = \bigotimes_{i=1}^d \mathcal{G}_t^{\alpha_i,\beta_i}$,
and we must estimate the supremum in $0<t\le 1$ of
\begin{equation} \label{max_int}
\left|\int_{(0,\pi)^d}  \mathcal{G}_t^{\alpha,\beta}(\theta,\varphi) f(\varphi)\,
	d\m_{\ab}(\varphi)\right|,  \qquad \theta \in (0,\pi)^d,
\end{equation}
with $f \in L^1(d\m_{\ab})$. Here we may assume that $f \ge 0$. 
If we split each coordinate interval $(0,\pi)$ in halves, the cube will be split in $2^d$
subcubes. It will be enough to consider $\theta$ in one of these, say $(0,\pi/2)^d$, since
the other subcubes can be treated like the one we select, if one places the Laguerre origin in
the argument below at each corner of the large cube.

Assuming thus $\theta \in (0,\pi/2)^d$,
we first do away with the integration over the set where $3\pi/4 < \varphi_i < \pi$ for some $i$. 
If $\theta_i< \pi/2$ and $\varphi_i > 3\pi/4$, it follows from Theorem \ref{thm:main} that 
$\mathcal{G}_t^{\alpha_i,\beta_i}(\theta_i,\varphi_i)\lesssim 1$. For any nonempty subset $\Lambda$
of $\{1,2,\ldots,d\}$, consider that part of the integral in \eqref{max_int} taken over the set
$$
\{\varphi: 3\pi/4 < \varphi_i < \pi \; \textrm{for}\; i \in \Lambda\; 
\textrm{and}\; 	0< \varphi_i < 3\pi/4 \; \textrm{for}\; i \notin \Lambda\}.
$$

Carrying out the integrations over the interval $(3\pi/4,\pi)$ in each variable $\varphi_i$ with 
$i \in \Lambda$, we can estimate this part of \eqref{max_int} by
$$
\int {\prod}' \mathcal{G}_t^{\alpha_i,\beta_i}(\theta_i,\varphi_i) f_{\Lambda}(\varphi')\,
	d\m'_{\ab}(\varphi').
$$
Here the product ${\prod}'$ is taken only over $i \in \{1,\ldots,d\}\backslash \Lambda$, and
$\varphi'$ consists only of the corresponding coordinates $\varphi_i$. Further, $f_{\Lambda}(\varphi')$
is the result of integrating the given function $f$ with respect to $d\m_{\alpha_i,\beta_i}$ 
in $(3\pi/4,\pi)$
for each $i \in \Lambda$. Finally, $d\m'_{\ab}$ is the product of those measures $d\m_{\alpha_i,\beta_i}$
with $i \notin \Lambda$, each restricted to $(0,3\pi/4)$.
 In the extreme case $\Lambda = \{1,2, \dots, d\}$, the expression
should be interpreted as $\int_{(3\pi/4, \pi)^d}f{d\m_{\alpha,\beta}}$, which certainly satisfies
the weak type (1,1) estimate, so this case can be neglected.

We thus arrive at an integral of the same type as that in \eqref{max_int} but in lower 
dimension and with the important difference that the integration is now only over $(0,3\pi/4)$
in each variable. Summing up, this means that when estimating \eqref{max_int}, we can restrict
the integration to the cube $(0,3\pi/4)^d$. What we must control is thus the operator
$$
Mf(\theta) = \sup_{0<t\le T}  \int_{(0, 3\pi/4)^d} \mathcal
G_t^{\alpha,\beta}(\theta, \varphi)f(\varphi)\,
d\m_{\alpha,\beta}(\varphi), \qquad \theta \in (0, \pi/2)^d,
$$
for $0\le f \in L^1(d\m_{\alpha,\beta}).$
We will show that the weak type $(1,1)$ of $M$ follows from the analogous property of a certain
Laguerre maximal operator, which was proved by the authors in \cite{NoSj1} to satisfy the weak type $(1,1)$
estimate in the appropriate measure space.

Assuming for a moment that $d=1$, we observe that by Theorem \ref{thm:main}
$$
\mathcal{G}_t^{\ab}(\theta,\varphi) \simeq \simeq (t+\theta \varphi)^{-\alpha-1/2}
	\frac{1}{\sqrt{t}} \exp\bigg( -c \frac{(\theta-\varphi)^2}{t}\bigg),
$$
uniformly in $\theta,\varphi \in (0,3\pi/4)$ and $0<t \le 1$. On the other hand, the one-dimensional
heat kernel associated with the Laguerre system $\{\psi_k^{\alpha}\}$ considered in \cite{NoSj1}
is expressed explicitly by
$$
K_t^{\alpha}(x,y) = \frac{1}{2\sinh t} \exp\bigg( -\frac{1}{4} \coth t \,\big(x^2+y^2\big)\bigg)
	(xy)^{-\alpha} I_{\alpha}\bigg( \frac{xy}{2\sinh t} \bigg)
$$
for $x,y,t>0$ and $\alpha>-1$; here $I_{\alpha}$ is the modified Bessel function of the first kind
and order $\alpha$. By means of the well-known asymptotics
$$
I_{\alpha}(z) \simeq z^{\alpha}, \quad z \to 0^+, \qquad \textrm{and} \qquad
I_{\alpha}(z) \simeq z^{-1/2} \exp(z), \quad z \to \infty,
$$
it is straightforward to check that
$$
K_t^{\alpha}(x,y) \simeq \simeq (t+xy)^{-\alpha-1/2} 
	\frac{1}{\sqrt{t}} \exp\bigg( -c \frac{(x-y)^2}{t}\bigg),
$$
uniformly in $x,y \in (0,3\pi/4)$ and $0< t \le T$, for any fixed $T>0$. 
Thus we see that there exists a constant $C>0$ such that
$$
\mathcal{G}_t^{\ab}(\theta,\varphi) \lesssim K^{\alpha}_{Ct}(\theta,\varphi), \qquad
	\theta,\varphi \in (0,3\pi/4), \quad 0< t \le 1.
$$
Moreover, the related Laguerre measure is $d\eta_{\alpha}(x) = x^{2\alpha+1}dx$ and hence
$d\mu_{\ab}(\theta) \simeq d\eta_{\alpha}(\theta)$ for $\theta \in (0,3\pi/4)$.

Coming back to arbitrary dimension $d \ge 1$ and taking into account the tensor product structures of 
the Jacobi and Laguerre settings, we infer that for some $C>0$
$$
\mathcal{G}_t^{\ab}(\theta,\varphi) \lesssim
	K_{Ct}^{\alpha}(\theta,\varphi), \qquad \theta,\varphi \in (0,3\pi/4)^d, \quad 0<t\le 1.
$$
Further, the corresponding Jacobi and Laguerre measures
are comparable on $(0,3\pi/4)^d$. This reveals that $M$ is controlled pointwise by the Laguerre
maximal operator
$$
K_*^{\alpha}f(x) = \sup_{t>0} \int_{(0,\infty)^d} K_t^{\alpha}(x,y) |f(y)|\, d\eta_{\alpha}(y)
$$
restricted to the cube $(0,3\pi/4)^d$. By \cite[Theorem 1.1]{NoSj1}, $K_*^{\alpha}$ satisfies
the weak type $(1,1)$ estimate with respect to the measure space $((0,\infty)^d,d\eta_{\alpha})$.
We conclude that $M$ is of weak type $(1,1)$ with respect to $((0,3\pi/4)^d,d\mu_{\ab})$, as desired. 

The proof of Theorem \ref{thm:max} is complete.
\end{proof}

\vspace{5pt}

%%%%%%%%%%%%%%%%%%%%%%%%%%%%%%%%%%%%%%%%%%%%%%%%%%%%%%%%%%%%%%%%%%%%%%%%%%%%%%%%%%%%%%%%%%%%%%%%%%%%%
\section*{Appendix: Poisson kernel estimates} \label{sec:poisson}
%%%%%%%%%%%%%%%%%%%%%%%%%%%%%%%%%%%%%%%%%%%%%%%%%%%%%%%%%%%%%%%%%%%%%%%%%%%%%%%%%%%%%%%%%%%%%%%%%%%%%

We complement the Jacobi heat kernel estimates by showing the following sharp bounds for the
Poisson-Jacobi kernel in the Jacobi trigonometric polynomial setting.
Clearly, this result can easily be transferred to the Jacobi trigonometric `function' setting.
In contrast to the preceding argument, the proof is based on an exact, 
positive representation of the kernel.
\begin{thm} \label{thm:Poisson}
Assume that $\alpha,\beta \ge -1/2$. Given any $T > 0$, we have
$$
\mathcal{H}_t^{\ab}(\theta,\varphi) \simeq \Big( t^2+ \theta^2+\varphi^2 \Big)^{-\alpha-1/2}
	\Big( t^2 + (\pi-\theta)^2 + (\pi-\varphi)^2 \Big)^{-\beta-1/2}
		\frac{t}{t^2+(\theta-\varphi)^2},
$$
uniformly in $\theta, \varphi \in [0,\pi]$ and $0 < t \le T$, and
$$
\mathcal{H}_t^{\ab}(\theta,\varphi) \simeq \exp\bigg( -t \frac{\alpha+\beta+1}{2} \bigg),
$$
uniformly in $\theta,\varphi \in [0,\pi]$ and $t \ge T$.
\end{thm}

The representation formula we shall use is \cite[Proposition 4.1]{NoSj2}, which says that for
$\ab \ge -1/2$
\begin{equation} \label{PJker}
\mathcal{H}_t^{\ab}(\theta,\varphi) = c_{\alpha,\beta} \,\sinh\frac{t}2
	\iint \frac{d\Pi_{\alpha}(u)d\Pi_{\beta}(v)}
	{(\cosh\frac{t}2-1 + q(\theta,\varphi,u,v))^{\alpha+\beta+2}},
\end{equation}
with $c_{\alpha,\beta} =2^{-\alpha-\beta-1}\slash \m_{\ab}(0,\pi)$ and 
$$
q(\theta,\varphi,u,v) = 1 - u \sin\frac{\theta}2 \sin\frac{\varphi}2
	- v \cos\frac{\theta}2 \cos\frac{\varphi}2, \qquad \theta,\varphi \in [0,\pi], \quad u,v \in [-1,1].
$$
This is based on the same product formula due to Dijksma and Koornwinder that we used in Section 
\ref{sec:prep}. The behavior of the double integral in \eqref{PJker} can be described by means
of the following technical result.
\begin{lem} \label{lem:int_est}
Let $\kappa \ge 0$ and $\gamma$ and $\nu$ be such that $\gamma > \nu +1/ 2 \ge 0$.
Then
$$
\int_{[-1,1]} \frac{d\Pi_{\nu}(s)}{(D-Bs)^{\kappa}(A-Bs)^{\gamma}} \simeq
	\frac{1}{(D-B)^{\kappa} A^{\nu+1/2} (A-B)^{\gamma-\nu-1/2}}, \qquad 0< B < A < D.
$$
\end{lem}

\begin{proof}
We may assume that $B=1$, since one can factor out a power of $B$ from both sides of the formula.
The case $\nu=-1/2$ is trivial, so it is enough to consider $\nu+1/2 > 0$. Observe that
$$
\int_{-1}^1 \frac{(1-s^2)^{\nu-1/2}\, ds}{(D-s)^{\kappa}(A-s)^{\gamma}} \simeq
	\int_0^1 \frac{(1-s)^{\nu-1/2}\, ds}{(D-s)^{\kappa}(A-s)^{\gamma}} =
		\int_0^1 \frac{u^{\nu-1/2} \, du}{(D-1 + u)^{\kappa}(A-1+u)^{\gamma}}.
$$
Thus it suffices to analyze the last integral, which we denote by $\mathcal{I}$.
Now $A>1$, and we consider the following two cases. 
\newline {\bf Case 1:} $\mathbf{A \ge 2.}$ Since in this case $D-1 + u \simeq D-1$ and
$A-1 + u \simeq A-1 \simeq A$ for $u \in (0,1)$, the conclusion is immediate.
\newline {\bf Case 2:} $\mathbf{1 < A < 2.}$
We split $\mathcal{I}$ as
$$
\mathcal{I} = \Bigg\{ \int_0^{A-1} + \int_{A-1}^1 \Bigg\} \;
	\frac{u^{\nu-1/2}\, du}{(D-1+u)^{\kappa}(A-1+u)^{\gamma}} \equiv \mathcal{I}_1 + \mathcal{I}_2.
$$
Then
$$
\mathcal{I}_1 \simeq \frac{1}{(D-1)^{\kappa}(A-1)^{\gamma}} \int_0^{A-1} u^{\nu-1/2}\, du
	\simeq \frac{1}{(D-1)^{\kappa}(A-1)^{\gamma-\nu-1/2}}
$$
and
$$
\mathcal{I}_2 \lesssim \frac{1}{(D-1)^{\kappa}} \int_{A-1}^1 u^{\nu-1/2-\gamma}\, du
	\lesssim \frac{1}{(D-1)^{\kappa}(A-1)^{\gamma-\nu-1/2}}.
$$
Since $A\simeq 1$, this implies the conclusion.
\end{proof}

\begin{proof}[Proof of Theorem \ref{thm:Poisson}]
We use \eqref{PJker} and apply Lemma \ref{lem:int_est} twice, first to the integral against
$d\Pi_{\beta}(v)$, with the parameters $\nu=\beta$, $\kappa = 0$, $\gamma = \alpha+\beta+2$,
$A=\cosh\frac{t}2-u\sin\frac{\theta}2 \sin\frac{\varphi}2$, $B=\cos\frac{\theta}2 \cos\frac{\varphi}2$,
and then to the resulting integral against $d\Pi_{\alpha}(u)$, with the parameters
$\nu=\alpha$, $\kappa = \beta+1/2$, $\gamma=\alpha+3/2$, $D=\cosh\frac{t}2$,
$A=\cosh\frac{t}2-\cos\frac{\theta}2 \cos\frac{\varphi}2$ and $B=\sin\frac{\theta}2 \sin\frac{\varphi}2$.
This leads to the estimates
\begin{align*}
& \mathcal{H}_t^{\ab}(\theta,\varphi)\\
& \simeq
	\frac{\sinh\frac{t}2}{\left(\cosh\frac{t}2-\cos\frac{\theta}2 \cos\frac{\varphi}2\right)^{\alpha+1/2}
	\left(\cosh\frac{t}2-\sin\frac{\theta}2 \sin\frac{\varphi}2\right)^{\beta+1/2}
	\left(\cosh\frac{t}2-\sin\frac{\theta}2 \sin\frac{\varphi}2-\cos\frac{\theta}2
	 \cos\frac{\varphi}2\right)} \\
& = \frac{1}{\big(\cosh\frac{t}2 -1 + 
	\sin^2\frac{\theta-\varphi}4+\sin^2\frac{\theta+\varphi}4\big)^{\alpha+1/2}} \\
& \qquad \times \frac{1}{\big(\cosh\frac{t}2 -1 + 
	\sin^2\frac{(\pi-\theta)-(\pi-\varphi)}4+\sin^2\frac{(\pi-\theta)+(\pi-\varphi)}4\big)^{\beta+1/2}}
 \;	\frac{\sinh\frac{t}2}{\cosh\frac{t}2-1 + 2\sin^2\frac{\theta-\varphi}4}.	
\end{align*}
As easily verified, the conclusion now follows.
\end{proof}

\vspace{5pt}

%%%%%%%%%%%%%%%%%%%%%%%%%%%%%%%%%%%%%%%%%%%%%%%%%%%%%%%%%%%%%%%%%%%%%%%%%%%%%%%%%%%%%%%%%%%%%%%%%%%%%%%%%

\end{document}